\documentclass{article}

\usepackage{amssymb}
\usepackage{amsthm}
\usepackage{amsfonts}
\usepackage{amsmath}
\usepackage[all]{xypic}
\usepackage{graphicx}
\usepackage{dlfltxbcodetips}
\usepackage{longtable}
\usepackage{enumerate}
\usepackage{fullpage}
\usepackage{multirow}

\newtheorem{thm}{Theorem}[section]

\newtheorem{prop}[thm]{Proposition}

\theoremstyle{remark}

\theoremstyle{definition}
\newtheorem{ex}{Example}[section]
\newtheorem{defn}{Definition}[section]

\newcommand{\Endex}{\hfill $\diamondsuit$}

\newcommand{\R}{\mathbb{R}}
\renewcommand{\S}{\mathbb{S}}
\newcommand{\posOrth}{\R^n_+}
\newcommand{\posOrthP}{\R^2_+}

\newcommand{\seq}[3]{\left\{#1\right\}_{#2}^{#3}}
\newcommand{\goaway}[1]{}
\newcommand{\tup}[1]{\left\langle #1 \right\rangle}

\newcommand{\smsetminus}{\stackrel{\smallsetminus}{\phantom{\_}}}

\newcommand{\myNorm}[1]{\left\|{#1}\right\|}
\newcommand{\soln}{\seq{x_n}{n=-k}{\infty}}

\newcommand{\xbar}{\bar{x}}
\newcommand{\Xbar}{\bar{\calX}}

\newcommand{\calX}{\mathcal{X}}

\newcommand{\Pbox}{P_\Box}

\renewcommand{\deg}{\mathrm{deg}}

\newcommand{\posCoeffs}{\textbf{PosCoeffs} }
\newcommand{\subPoly}{\textbf{SubPoly} }
\newcommand{\lCoeff}{\textbf{LCoeff} }
\newcommand{\Const}{\textbf{Const} }

\title{A New Algorithm for Proving Global Asymptotic Stability of Rational Difference Equations}
\author{Emilie Hogan and Doron Zeilberger}

\begin{document}

\maketitle

\abstract{
Global asymptotic stability of rational difference equations is an area of research that
has been well studied. In contrast to the many current methods for proving global
asymptotic stability, we propose an algorithmic approach. The algorithm we summarize here
employs the idea of contractions. Given a particular rational difference equation,
defined by a function $Q: \R^{k+1} \rightarrow \R^{k+1}$, we attempt to find a $K$ value
for which $Q^K$ shrinks distances to the difference equation's equilibrium point. We
state some general results that our algorithm has been able to prove, and also mention
the implementation of our algorithm using Maple.}

\section{Introduction}

In this paper we will introduce an algorithmic approach to proving \emph{global
asymptotic stability (GAS)} of equilibrium points of rational difference equations. The
field of rational difference equations has applications to many other fields including
biology, economics, and dynamical systems. In application areas, one often studies
time-evolving sequences produced by recurrences with the goal of discovering end behavior
of the sequence, given some initial conditions. There are many types of end behavior that
one may be interested in. We will be concerned only with global asymptotic stability.
Essentially, given a fixed rational difference equation, when the sequence that it
produces converges for \emph{any} reasonable initial conditions, we say that it is GAS.
In Section \ref{DefnsGAS}, we will state the precise definition for GAS, as well as
introduce all of the other definitions necessary to study stability of difference
equations. We will also state a theorem, originally proved in \cite{KrNNeT1999}, which
will be the basis of our algorithm. The main algorithm is presented in Sections
\ref{GAStoPoly} and \ref{PosMethods}. In the algorithm we first reduce the problem of GAS
to the problem of proving that a particular polynomial is positive. Then, we prove that a
multivariate polynomial is positive (when all of its variables are taken to be positive)
using our new algorithm. Next, Section \ref{ProofOfConcept} contains a proof-of-concept
that our algorithm is indeed applicable to prove GAS. In Section \ref{GASresults} we
state a few of the results that our algorithm can prove. In addition, in Section
\ref{MapleCodeGAS}, we mention the most useful commands in the Maple package that
accompanies this paper.

\subsection{Definitions}\label{DefnsGAS} Following the various works of Ladas, et. al.
\cite{CaELaG2008,KoVLaG1993,KuMLaG2001}, we begin by stating a few standard definitions
needed to study rational difference equations and stability.
\begin{defn}\label{DEdefn}
A \emph{rational difference equation} (of \emph{order} $k+1$) is an equation of the form
\begin{align}\label{DEform}
x_{n+1} &= R(x_{n},x_{n-1},\ldots,x_{n-k})
\end{align}
where the function $R(u_0,u_1,\ldots,u_k)$ is a rational function which maps $I^{k+1}$ to
$I$, for some interval $I \subseteq \R$. Typically, we will take $I$ to be $[0,\infty)$
or $(0,\infty)$.
\end{defn}

Given a function $R$ we say that a \emph{solution} of \eqref{DEform} is a sequence
$\seq{x_n}{n=-k}{\infty}$ which satisfies \eqref{DEform}. One can also think of a
solution, $\seq{x_n}{n=-k}{\infty}$, as being associated to the specific initial
conditions $\{x_{-k},\ldots,x_{0}\}$ created by repeatedly applying  $R$. If a solution
is constant, $x_n=\xbar$, for all $n\geq -k$ then we say that the solution is an
\emph{equilibrium solution}, and $\xbar$ is called an \emph{equilibrium point}, or simply
an \emph{equilibrium} of $F$. In practice, we find the equilibria by solving the equation
$\xbar = R(\xbar,\ldots,\xbar)$, and taking the solutions which lie in the interval $I$.

The main topic to be investigated in this paper is end behavior, specifically
\emph{stability}, of a solution of a given difference equation. There are various notions
of stability that will now be defined.

\begin{defn}\label{StabilityDefn}
An equilibrium point, $\xbar$, of \eqref{DEform} is said to be
\begin{enumerate}
\item \emph{locally stable} if for every $\varepsilon>0$ there exists $\delta>0$ such
    that if $\soln$ is a solution to \eqref{DEform} with the property that
    \[|x_{-k}-\xbar|+|x_{-k+1}-\xbar|+\cdots+|x_0-\xbar| < \delta\]
    then $|x_n-\xbar|<\varepsilon$ for all $n\geq 0$.
\item \emph{locally asymptotically stable (LAS)} if $\xbar$ is locally stable, and if
    there exists a $\gamma>0$ such that if $\soln$ is a solution to \eqref{DEform}
    with the property that
    \[|x_{-k}-\xbar|+|x_{-k+1}-\xbar|+\cdots+|x_0-\xbar| < \gamma\]
    then
    \[\lim_{n\rightarrow\infty} x_{n} = \xbar\]
\item a \emph{global attractor} if for every solution, $\soln$, of \eqref{DEform} we
    have
    \[\lim_{n\rightarrow\infty} x_{n} = \xbar\]
\item \emph{globally asymptotically stable (GAS)} if $\xbar$ is a global attractor,
    and $\xbar$ is locally stable.
\item \emph{unstable} if $\xbar$ is not locally stable.
\end{enumerate}
\end{defn}

Our goal in this paper is to present an algorithm to prove GAS. Since GAS implies LAS,
the first step must be to prove LAS (since, if a difference equation is not LAS it can't
be GAS). The \emph{linearized stability theorem}, which provides easily verifiable
criteria for local asymptotic stability, can be found in many books and papers
\cite{CaELaG2008, ElS2000, HaJKoH1991, KuMLaG2001, MaM1999}. Because it is not central to
our algorithm, we will omit the theorem and notation needed to state it.

In contrast to local asymptotic stability which is relatively easy to verify using the
linearized stability theorem, global asymptotic stability has no similarly general
necessary and sufficient conditions. There are a handful of theorems, providing
sufficient conditions, that have been used to verify the global asymptotic stability of
many specific difference equations. However, given a difference equation defined by the
function $R$, it is not always obvious which theorem to apply. For a discussion of many
of these theorems see \cite{CaELaG2008}.

Our algorithm will only rely on the following theorem which is first presented in a paper
by Kruse and Nesemann \cite{KrNNeT1999}. It will be stated it in a slightly different
manner than it appears in their paper, using the notation we have established in this
paper. First, it will be necessary to consider the difference equation associated to a
function $R$ in vector form. Let $Q:I^{k+1}\rightarrow I^{k+1}$ be defined from $R$ as
\begin{align}\label{GvectMap}
Q(\calX_{n}) = Q\left(\left[\begin{array}{c}
                             x_{n}\\
                             x_{n-1}\\
                             \vdots\\
                             x_{n-k}
                             \end{array}\right]\right) = \left[\begin{array}{c}
                                                         R(x_{n},\ldots,x_{n-k})\\
                                                         x_{n}\\
                                                         \vdots\\
                                                         x_{n-k+1}
                                                         \end{array}\right] = \calX_{n+1}.
\end{align}
Note that this transformation from $R$ to $Q$ essentially creates an order 1 mapping out
of an order $k+1$ mapping. In addition, $Q$ is now a map that can be composed with
itself, so
\[\calX_{n} = Q^n(\calX_0)\]
where $\calX_0=\tup{x_0,\ldots,x_{-k}}$ is the vector of initial conditions. Now we can
state the theorem.

\begin{thm}[Kruse, Nesemann 1999]\label{contrGAS}
Let $\myNorm{\cdot}$ denote the Euclidean norm (i.e., $\myNorm{\tup{a,b}} =
\sqrt{a^2+b^2}$). Let $\S$ denote either $[0,\infty)$ or $(0,\infty)$ (the function $Q$
will necessitate which). Let $Q : \S^{k+1} \rightarrow \S^{k+1}$ be a continuous mapping
of the form \eqref{GvectMap} with a unique fixed point $\Xbar \in \S^{k+1}$. Suppose for
the discrete dynamic system
\begin{align}\label{vectRecur}
\calX_{n+1} = Q(\calX_{n}), \quad n=0,1,2,\ldots
\end{align}
there exists an integer $K \geq 1$ such that the $K^{th}$ iterate of $Q$ satisfies
\begin{align}\label{contrCriteria}
\myNorm{Q^K(\calX) - \Xbar} < \myNorm{\calX-\Xbar} \quad \text{for all } \calX \in \S^{k+1}, \calX \neq \Xbar.
\end{align}
Then $\Xbar$ is GAS with respect to the norm $\myNorm{\cdot}$.
\end{thm}
First, notice that this integer $K$ tells us which power of $Q$ is a contraction with
respect to $\Xbar$, i.e., $Q^K$ shrinks distances to $\Xbar$. This gives an intuitive
reason for $\myNorm{Q^K(\calX) - \Xbar} < \myNorm{\calX-\Xbar}$ to imply global
asymptotic stability. The proof of this theorem can be found in \cite{KrNNeT1999}.

Notice that the various definitions of stability, as they were stated in Definition
\ref{StabilityDefn}, do not quite apply here because our unique fixed point (or
equilibrium) is a vector rather than a scalar. However, Definition \ref{StabilityDefn},
can be easily translated to the vector case. The recurrence is \eqref{vectRecur}, the
equilibrium is a vector solution to the equation $Q(\Xbar)=\Xbar$, and the order of the
recurrence, $k+1$, is 1 (so $k=0$). Other than these minor changes, a word for word
translation of Definition \ref{StabilityDefn} is what we mean by $\Xbar$ being GAS in
Theorem \ref{contrGAS}.

Next we will see how we utilized this theorem to create a global asymptotic stability
proof algorithm.

\section{From Global Asymptotic Stability to Polynomial Positivity}\label{GAStoPoly} In
this section we will see how to reduce the question of global asymptotic stability of a
rational difference equation to a question about an associated polynomial being positive.
Throughout this section assume that we have fixed a rational difference equation,
\begin{align}\label{theRDE}
x_{n+1} = R(x_n,\ldots,x_{n-k}),
\end{align}
of order $k+1$, with a unique equilibrium $\xbar$. Also assume that $R$ is a rational
function with positive coefficients, so $R: [0,\infty)^{k+1} \rightarrow [0,\infty)$, and
$\xbar$ is non-negative (if there is no constant term in the denominator of $R$ we cannot
allow 0 to be in the domain, so $R:(0,\infty)^{k+1} \rightarrow (0,\infty)$, and $\xbar$
must be strictly positive). In order to apply Theorem \ref{contrGAS} we must think of
\eqref{theRDE} and its equilibrium in their vector forms. For example, if $x_{n+1} =
\frac{4+x_n}{1+x_{n-1}}$ then
\begin{align*}
Q\left(\left[\begin{array}{c}
               x_n\\
               x_{n-1}
               \end{array}\right]\right) &= \left[\begin{array}{c}
                                        \frac{4+x_n}{1+x_{n-1}}\\
                                        x_{n}
                                        \end{array}\right],
\end{align*}
and $\Xbar = \tup{2,2}$. In this case $k=1$, so $R:[0,\infty)^2 \rightarrow [0,\infty)$,
and $Q:[0,\infty)^2\rightarrow [0,\infty)^2$.

The goal will be to find a positive integer, $K$, which satisfies
\eqref{contrCriteria}\goaway{, i.e., $\myNorm{Q^K(\calX)-\Xbar} < \myNorm{\calX-\Xbar}$}.
Motivated by this goal, we will construct the following polynomial, given specific $Q$,
$\Xbar$, and $K$ (assume we have conjectured some value for $K$):
\begin{align}\label{polyToPos}
P_{Q,\Xbar,K}(\calX) = \text{numerator}\left( \myNorm{\calX-\Xbar}^2 - \myNorm{Q^K(\calX) - \Xbar}^2 \right).
\end{align}
Consider the implication of $P_{Q,\Xbar,K}>0$ for $\calX \geq 0$ (or $>0$, both
componentwise), and $\calX \neq \Xbar$.
\begin{align}
               &\quad 0 < \text{numerator}\left( \myNorm{\calX-\Xbar}^2 - \myNorm{Q^K(\calX) - \Xbar}^2 \right)\notag\\
\Longleftrightarrow&\quad 0 <\myNorm{\calX-\Xbar}^2 - \myNorm{Q^K(\calX) - \Xbar}^2\notag\\
\Longleftrightarrow&\quad \myNorm{Q^K(\calX) - \Xbar}^2 < \myNorm{\calX-\Xbar}^2\notag\\
\Longleftrightarrow&\quad \myNorm{Q^K(\calX) - \Xbar} < \myNorm{\calX-\Xbar}.\label{contrCriteria2}
\end{align}
Of course, the first implication, undoing the numerator from line 1 to line 2, in general
will not preserve an inequality since the denominator may be negative. However, because
we are squaring the Euclidean norm, the common denominator is always a product of sums of
squares. Taking the numerator is then equivalent to multiplying both sides by the
denominator, a positive quantity, which will not change the direction of the inequality.
Notice that the final implicant, \eqref{contrCriteria2}, is simply \eqref{contrCriteria},
so proving $P_{Q,\Xbar,K}>0$ for some $K$ implies that $\xbar$ is GAS for the rational
difference equation $R$. An algorithm for proving positivity will be shown in Section
\ref{PosMethods}. Also note that whenever the function $Q$ and equilibrium are clear from
context, they will be omitted from the subscript of $P$.

For a given $Q$ and $\xbar$ we know that showing positivity of an associated polynomial
implies GAS of $\xbar$ for $R$. We also know, given $K$, what that polynomial associated
to $Q$ and $\Xbar$ is. However, we still need to see how to conjecture a reasonable value
for $K$, and then how to prove that the polynomial is indeed positive. We will see how to
prove positivity in the next section. Now let's see how to conjecture a reasonable $K$
value given $R$ and $\xbar$ using a brute force method. Start with $K=1$ and apply the
following algorithm:
\begin{enumerate}
\item \label{MakePoly} Create the polynomial $P_{Q,\Xbar,K}(\calX)$
\item \label{MinTechnique} Apply a minimization technique to the polynomial
    $P_{Q,\Xbar,K}(\calX)$ (e.g., simulated annealing, gradient descent,
    Metropolis-Hastings algorithm, etc.) many times to find approximate local minima
    of $P_{Q,\Xbar,K}$.
\item
\begin{enumerate}
\item \label{FoundGoodK} If all approximate local minima are positive then
    conjecture that this $K$ works.
\item \label{KNotGood}If there is a negative minima then increment $K$ by 1 and
    go back to step 1.
\end{enumerate}
\end{enumerate}
For ease of computation, and since this is only to conjecture a $K$, we apply the
minimization technique in step \ref{MinTechnique} to a discrete set of points. We will
restrict to a fine mesh with large upper bound. For example, the cartesian product
$\bigtimes_{i=1}^{k+1} \left\{\varepsilon,2\varepsilon,\ldots,N\varepsilon\right\}$, for
some large value of $N$ and small value of $\varepsilon$. Then every point in the mesh is
a vector of the form $\tup{i_1\varepsilon, i_2\varepsilon,\ldots,i_{k+1}\varepsilon}$,
where $1 \leq i_j \leq N$.

Note that this is not the only possible algorithm for conjecturing a value for $K$.
However, the main result in this paper is a positivity algorithm, so we will not consider
other possible algorithms. One could, in theory, replace step 2 with the following,
\emph{``Apply the polynomial positivity algorithm found in Section \ref{PosMethods}"}.
Then step 3 would become, \emph{``If the algorithm in step 2 fails, increase $K$ by 1 and
go back to step 1, otherwise return $K$"}. Using this positivity algorithm, once a $K$
value is found, it is also proved to be correct. However, using positivity in step 2 is
sometimes not feasible since it often takes more computer memory than the conjecturing
algorithm.

\section{An Algorithm to Prove Positivity of a Multivariate
Polynomial}\label{PosMethods} So far, our algorithm to prove global asymptotic stability
of a particular rational difference equation has reduced the problem to proving that an
associated polynomial is positive. Now the question becomes, how does one prove
positivity? In general one can show polynomial positivity using calculus, or using
cylindrical algebraic decomposition \cite{ArDCoGMcS1984}. However, both of these methods
do not work particularly well when the polynomial has very high degree. This is typically
the case for the polynomials produced in Section \ref{GAStoPoly}, so we must use a
different algorithm. We propose a new algorithm which was inspired by the following
definition and theorem found in \cite{HoHDaJ1998}.
\begin{defn}
The polynomial $P \in \R[x_1,\ldots,x_n]$ is
\begin{itemize}
\item \emph{positive} (resp. \emph{non-negative}) from $\mu$ iff $\forall x_1 \geq
    \mu,\ldots,x_n\geq\mu$, $P(x_1,\ldots,x_n)>0$ (resp. $P(x_1,\ldots,x_n)\geq 0$).
\item \emph{absolutely positive} (resp. \emph{absolutely non-negative}) from $\mu$
    iff $P$ is positive (resp. non-negative) from $\mu$, and every partial derivative
    (of any order), $P^*$, of $P$ is non-negative from $\mu$, i.e., $\forall x_1 \geq
    \mu, x_2 \geq \mu, \ldots, x_n \geq \mu$, $P^*(x_1,\ldots,x_n)\geq 0$.
\end{itemize}
\end{defn}
In addition, we will denote by $\sigma_{\mu_1,\ldots,\mu_n}(P)$ the polynomial obtained
from $P$ by translating in dimension $i$ by $\mu_i$ in the negative direction. In other
words, replace $x_i$ by $x_i+\mu_i$ in $P$ for all $i$. If $\mu_i = \mu$ for all $i$ then
we simply write $\sigma_\mu(P)$.Also from \cite{HoHDaJ1998}, a theorem that gives a
necessary and sufficient condition for absolute positivity (and absolute non-negativity)
is reproduced here.

\begin{thm}[Hong, Jaku\v{s} 1998]\label{absPosThm}
Let $P$ be a non-zero polynomial. Then $P$ is absolutely positive (resp. absolutely
non-negative) from $\mu$ iff every coefficient in $\sigma_\mu(P)$ is positive, and the
constant term is nonzero (resp. non-negative). In particular, if $\mu=0$ then every
coefficient in $P$ is positive and the constant term is nonzero (resp. non-negative).
\end{thm}

Now, it is certainly too much to hope for the polynomials $P_{Q,\Xbar,K}$ to be
\emph{absolutely} positive from zero. Of course, to satisfy Theorem \ref{contrGAS}, it is
only necessary that they be positive from zero (and possibly zero at a few points). Our
algorithm will subdivide the positive \emph{orthant} (the region in which all the
variables are non-negative), denoted by $\posOrth$ where $n$ is the number of variables
in $P$, into regions in which $P$ is positive on the boundary of the region, and
essentially absolutely positive in some direction away from the boundary (i.e., there is
a direction such that the directional derivative is positive).

Since the polynomials we construct while pursuing global asymptotic stability are
typically very complicated (high degree in all variables, some negative coefficients), we
cannot easily show that a directional derivative is positive. Instead, for each region $S
\subset \posOrth$ we will create a polynomial, $P_S(y)$ with the property that if
$P_S(y)\geq 0$ for all $y \in \posOrth$ then $P(x)\geq 0$ for all $x \in S$. We will
first describe the algorithm in two dimensions, and later generalize to the
$n$-dimensional case.

Let $P:=P(x,y)$ be a polynomial in two variables ($n=2$). In order to show that $P(x,y)
\geq 0$ for $(x,y) \in \posOrthP$, we first cut the positive quadrant into 4 regions as
shown in Figure \ref{fourRegions}, where $\xbar$ is some positive number. In the case
that $P=P_{Q,\Xbar,K}$ as in Section \ref{GAStoPoly}, $\xbar$ will be the equilibrium
point of the rational difference equation used to create $P$.
\begin{figure}[h]
\centering
    \includegraphics[scale=.45]{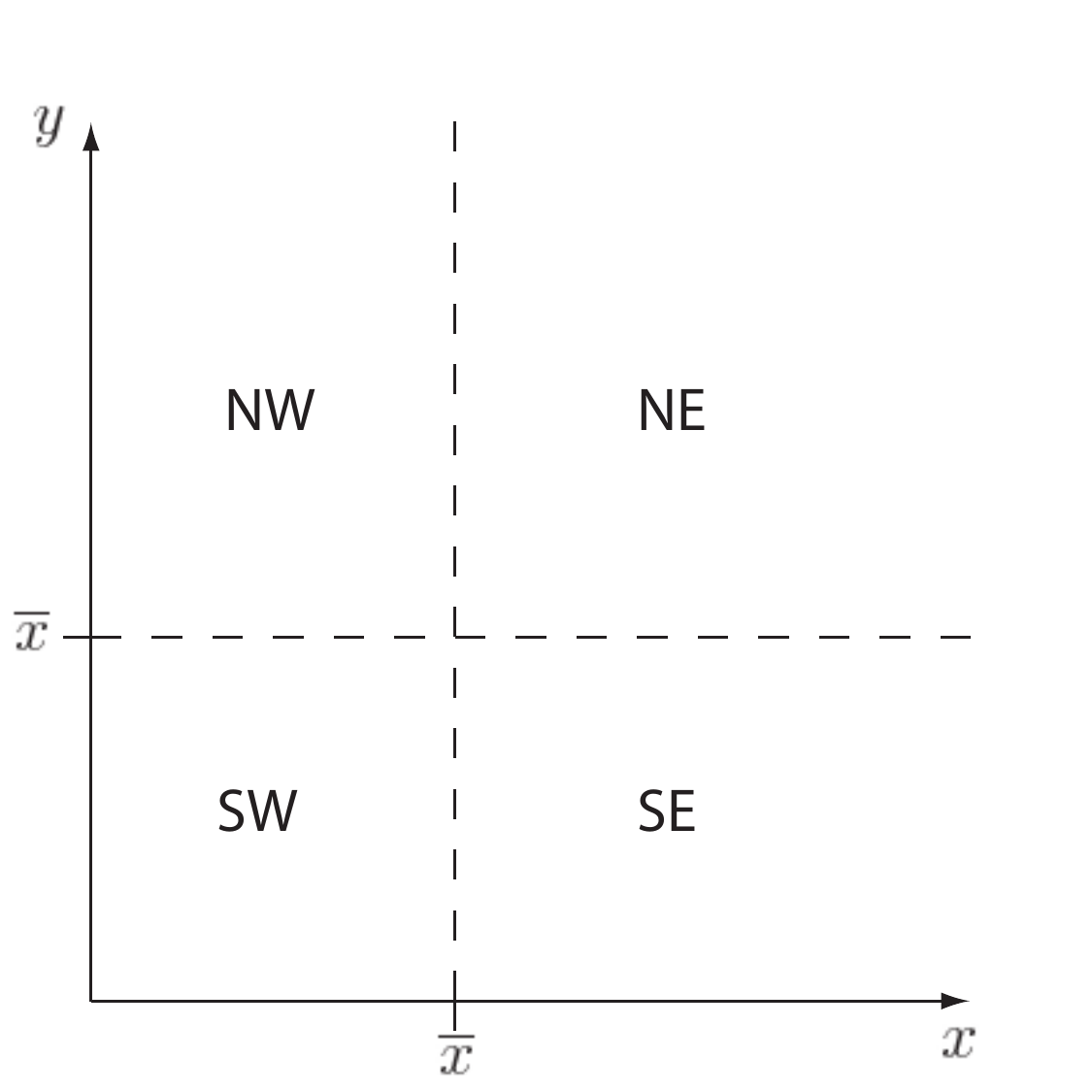}
    \caption{Cutting $\posOrthP$ into 4 regions}\label{fourRegions}
\end{figure}

For each of these four regions we create a new polynomial from $P$ by transforming the
region into $\posOrthP$, and making the corresponding variable substitutions. See Figures
\ref{NEmove} - \ref{NWmove} for the region transformations (transforming SE is analogous
to transforming NW by permuting the $x$ and $y$ axes).
\begin{figure}[h]
\centering
    \includegraphics[scale=.45]{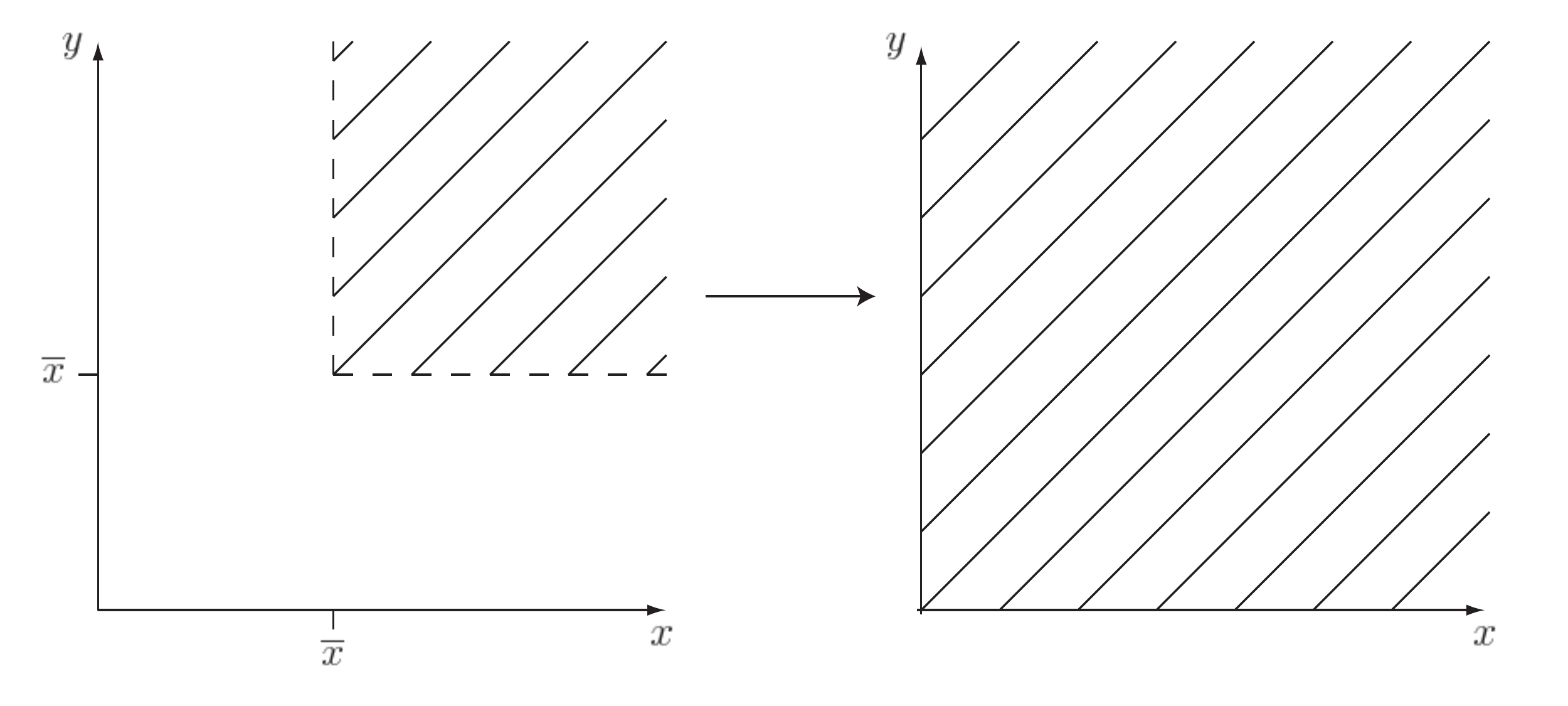}
\caption{Transforming NE to $\posOrthP$}\label{NEmove}
\end{figure}
\begin{figure}[h]
\centering
    \includegraphics[scale=.45]{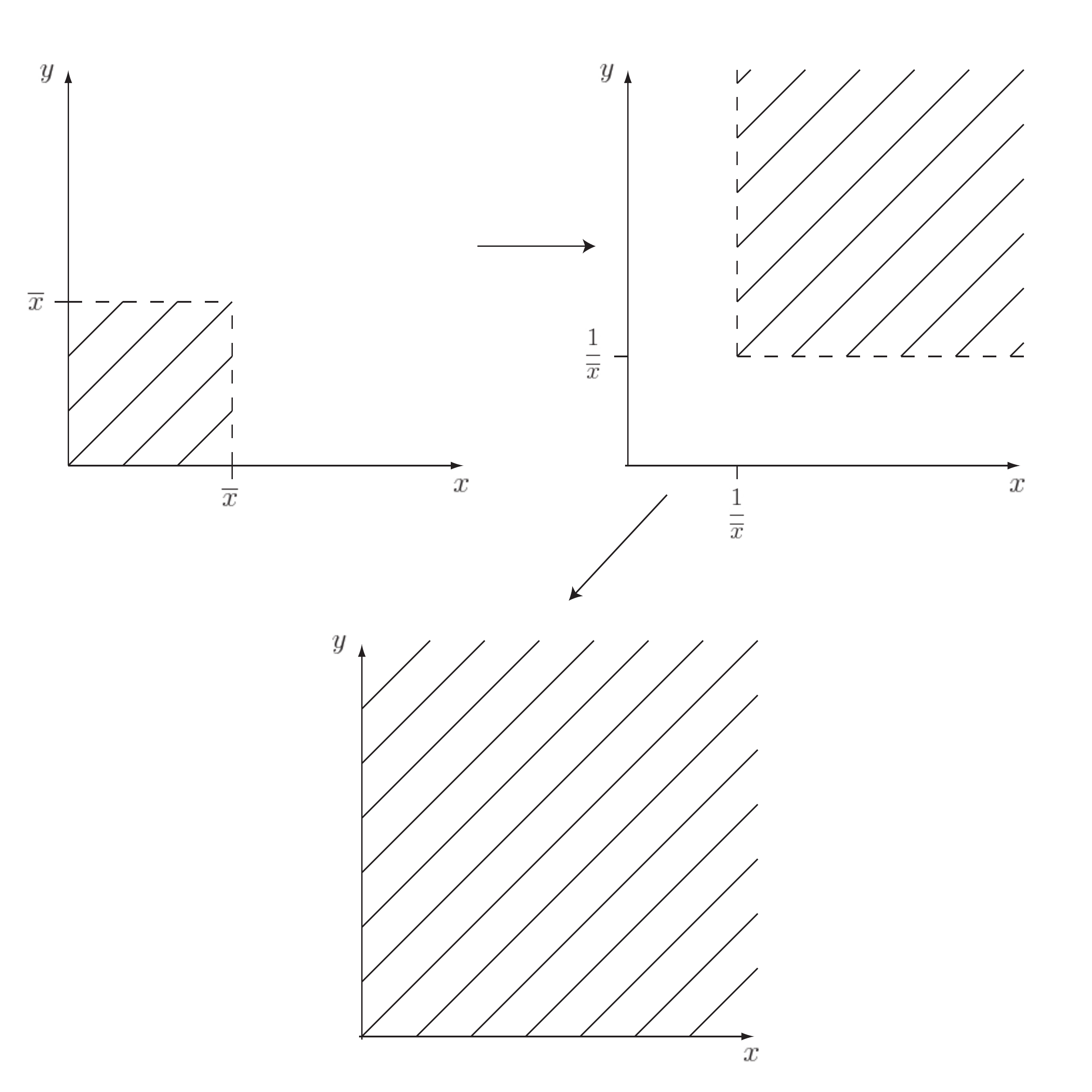}
\caption{Transforming SW to $\posOrthP$}\label{SWmove}
\end{figure}
\begin{figure}[h]
\centering
    \includegraphics[scale=.45]{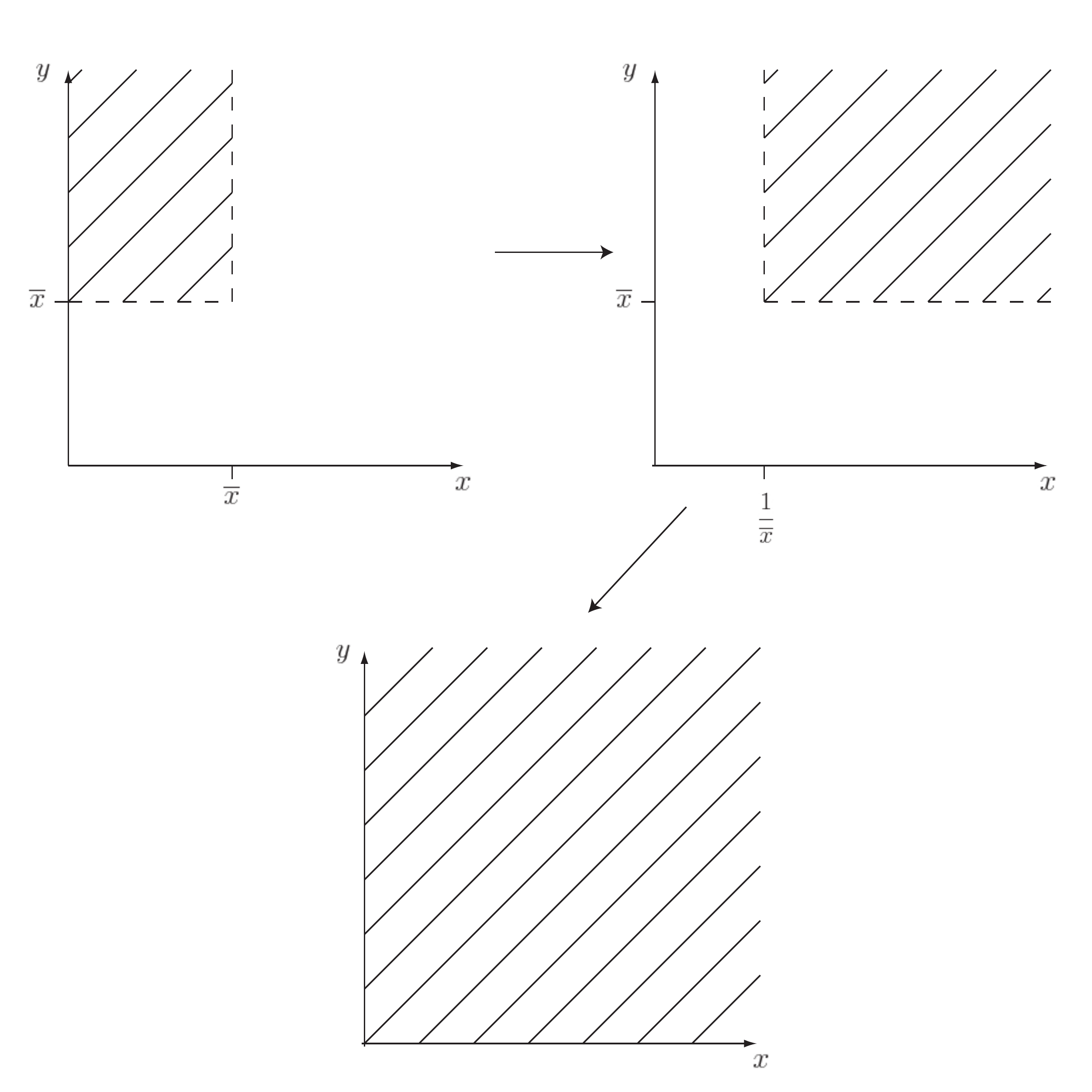}
\caption{Transforming NW to $\posOrthP$}\label{NWmove}
\end{figure}
Based on these region transformations we see that the associated polynomials are given by
\begin{align}\label{fourPolys}
P_{NE}(x,y) &= \sigma_{\xbar}(P) = P(x+\xbar,y+\xbar),\notag\\
P_{SW}(x,y) &= \sigma_{\frac{1}{\xbar}}\left(P\left( \frac{1}{x},\frac{1}{y} \right) x^{d_x} y^{d_y}\right)\notag \\
            &= P\left(\frac{1}{x+\frac{1}{\xbar}},\frac{1}{y+\frac{1}{\xbar}}\right)\left(x+\frac{1}{\xbar}\right)^{d_x}\left(y+\frac{1}{\xbar}\right)^{d_y},\\
P_{NW}(x,y) &= \sigma_{\frac{1}{\xbar},\xbar}\left( P\left(\frac{1}{x}, y\right) x^{d_x} \right) =
                P\left(\frac{1}{x+\frac{1}{\xbar}}, y+\xbar \right) \left( x+\frac{1}{\xbar}\right)^{d_x},\notag\\
P_{SE}(x,y) &= \sigma_{\xbar, \frac{1}{\xbar}}\left( P\left( x,\frac{1}{y} \right) y^{d_y} \right) =
                P\left(x+\xbar,\frac{1}{y+\frac{1}{\xbar}} \right) \left( y+\frac{1}{\xbar}\right)^{d_y}.\notag
\end{align}
The change of variables in $P$ before applying $\sigma$, inverting the variables or not,
is self explanatory based on the transformation of the associated region. However, we
must also multiply by $x^{d_x}$ and/or $y^{d_y}$ (where $d_z=$ the degree of $z$ in $P$
for $z=x,y$) as needed before applying the $\sigma$ shift operator so that the resulting
$\Pbox$ is still a polynomial. When talking generally about one of these polynomials, we
will denote it by $\Pbox$, where the $\Box$ can refer to an arbitrary region. Note that
if $\xbar=0$ we will only consider the $NE$ region, thus avoiding translating by
$\frac{1}{\xbar}=\frac{1}{0}$.

Before we continue the algorithm by giving criteria to test positivity of the polynomials
in \eqref{fourPolys} we must see why proving positivity of all $\Pbox$ will be enough to
prove positivity of $P(x,y)$ for all $(x,y) \in \posOrthP$.
\begin{prop}\label{pos4regions}
Let $P(x,y)$ be a polynomial, $d_x=\deg_{x}(P)$, $d_y=\deg_{y}(P)$, and $\xbar > 0$.
Consider the polynomials $P_{NE}, P_{SW}, P_{NW}, P_{SE}$ as defined in
\eqref{fourPolys}. If these four polynomials are all non-negative from 0 then $P(x,y)$ is
non-negative from 0.
\end{prop}
\begin{proof} For each of the four polynomials we will see that positivity for $(x,y) \in \posOrthP$
implies positivity of $P(x,y)$ in the corresponding region.
\begin{description}
\item[If $P_{NE}(x,y) \geq 0$ for $(x,y) \in \posOrthP$:] Then by definition of
    $P_{NE}(x,y)$ we have
    \begin{align*}
    P_{NE}(x,y)= P(x+\xbar,y+\xbar) &\geq 0 \quad \text{for } x\geq 0, y\geq 0.
    \end{align*}
    Let $x' := x+\xbar$ and $y' := y+\xbar$, then
    \begin{align*}
    P(x',y') &\geq 0 \quad \text{for } x'=x+\xbar \geq \xbar, \text{ and } y'=y+\xbar \geq \xbar.
    \end{align*}
    This says precisely that $P(x,y) \geq 0$ in the region $NE$.
\item[If $P_{SW}(x,y) \geq 0$ for $(x,y) \in \posOrthP$:] Again, by definition of
    $P_{SW}(x,y)$
    \begin{align*}
    P_{SW}(x,y) &= P\left(\frac{1}{x+\frac{1}{\xbar}},\frac{1}{y+\frac{1}{\xbar}}\right)\left(x+\frac{1}{\xbar}\right)^{d_x}\left(y+\frac{1}{\xbar}\right)^{d_y} \geq 0\\
                & \hspace{2.5in}\text{for } x\geq 0, y\geq 0.
    \end{align*}
    Following the previous case we first substitute $x':=x+\frac{1}{\xbar}$ and
    $y':=y+\frac{1}{\xbar}$ to get
    \begin{align*}
    P\left(\frac{1}{x'},\frac{1}{y'}\right)\left(x'\right)^{d_x}\left(y'\right)^{d_y} \geq 0
            \quad \text{for } x'=x+\frac{1}{\xbar}\geq\frac{1}{\xbar}, \text{ and } y'=y+\frac{1}{\xbar} \geq \frac{1}{\xbar}.
    \end{align*}
    Since we are only interested in the region for which $x'$ and $y'$ are both
    strictly positive we may cancel the $\left(x'\right)^{d_x}\left(y'\right)^{d_y}$
    without reversing the inequality. We also make a second substitution letting $x''
    := \frac{1}{x'}$ and $y'' := \frac{1}{y'}$. Now we see that
    \begin{align*}
    P(x'',y'') \geq 0 \quad \text{for } 0< x'' = \frac{1}{x'} \leq \xbar, \text{ and } 0< y'' = \frac{1}{x'} \leq \xbar
    \end{align*}
    which is simply $P(x,y)\geq 0$ in the region $SW$.
\item[If $P_{NW}(x,y) \geq 0$ for $(x,y) \in \posOrthP$:] From the definition in
    \eqref{fourPolys} this means
    \begin{align*}
    P_{NW}(x,y) &= P\left(\frac{1}{x+\frac{1}{\xbar}}, y+\xbar \right) \left( x+\frac{1}{\xbar}\right)^{d_x} \geq 0 \quad \text{for } x\geq 0, y\geq 0.
    \end{align*}
    As in the $SW$ case, we will make two substitutions. The first being $x' :=
    x+\frac{1}{\xbar}$ and $y' := y+\xbar$. This gives us
    \begin{align*}
    P\left(\frac{1}{x'}, y' \right) \left( x'\right)^{d_x} \geq 0 \quad \text{for } x'=x+\frac{1}{\xbar}\geq\frac{1}{\xbar}, \text{ and } y'=y+\xbar \geq \xbar.
    \end{align*}
    Again, we may cancel the $\left( x'\right)^{d_x}$ without reversing the
    inequality since $x'$ is strictly positive in the region in question. Finally, we
    make our second substitution, $x'' := \frac{1}{x'}$ (there is no second
    substitution for $y'$) which yields
    \begin{align*}
    P(x'',y') \geq 0 \quad \text{for } 0< x'' = \frac{1}{x'} \leq \xbar, \text{ and } y' \geq \xbar.
    \end{align*}
    Therefore, $P(x,y) \geq 0$ in the region $NW$.
\item[If $P_{SE}(x,y) \geq 0$ for $(x,y) \in \posOrthP$:] This case is analogous to
    the $NW$ case by interchanging the roles of $x$ and $y$.
\end{description}
In each of the four cases positivity of the polynomial corresponds to positivity of
$P(x,y)$ in the corresponding region.
\end{proof}

To prove positivity of each of the $\Pbox$ we will test two criteria, neither using
anything more powerful than high school algebra.
\begin{description}
\item[PosCoeffs:] From Theorem \ref{absPosThm}, if all coefficients, including the
    constant term, of $\Pbox$ are non-negative then $\Pbox(x,y) \geq 0$ for $(x,y)
    \in \posOrthP$.
\item[SubPoly:] If the only negative coefficient in $\Pbox$ (including the constant
    term) is on the $xy$ term then we check whether the binary quadratic form,
    \begin{align}\label{subP1}
    a x^2+b xy+c y^2
    \end{align}
    where $a$, $b$, and $c$ are coefficients of their respective terms in $\Pbox$, is
    positive definite (i.e., is positive for all $(x,y) \neq 0$) using its
    discriminant. The binary quadratic form discriminant of \eqref{subP1} is defined
    to be $d=4ac-b^2$ \cite{OMeO1971}. If $a, d > 0$, then \eqref{subP1} is positive.
    Then, if this ``sub-polynomial" of $\Pbox$ is positive, $\Pbox$ itself is
    positive (since the other coefficients are positive). Notice that this may not be
    the discriminant most are familiar with. A further discussion of why this is
    taken to be the discriminant can be found when the $n$-dimensional positivity
    algorithm is summarized later in this section.
\end{description}

We also have an easy way to test whether $P(x,y)<0$ for some $(x,y) \in \posOrthP$ by
checking the leading coefficient (the coefficient on the highest degree term) and
constant term.
\begin{description}
\item[LCoeff:] The leading coefficient must be positive, otherwise the polynomial
    eventually tends to negative infinity in some direction.
\item[Const:] Similarly, the constant term must be positive, otherwise the polynomial
    is negative in a neighborhood of the origin.
\end{description}

For each region $\Box$, if $\Pbox$ passes one of \posCoeffs or \subPoly then, by
Proposition \ref{pos4regions}, $P(x,y) \geq 0$ in the region $\Box$. If $\Pbox$ fails one
of \lCoeff or \Const then we output \textbf{false} immediately because we know that there
are points in region $\Box$ for which $P(x,y)$ is negative. However, for some region
$\Box$, if $\Pbox$ has too many negative coefficients, its leading coefficient is
positive, and its constant term is positive, then we must do more tests to establish
positivity of $\Pbox$ on $\posOrthP$.

We would like to subdivide our original region (NE, NW, SE, or SW) into finitely many
pieces and try again. However, there isn't an obvious way to do this since, except for
SW, the regions are infinite, and we have used our ``obvious" cutpoint, $\xbar$. So
instead, we will first map the infinite region into a finite rectangle with lower left
corner at the origin (see figures \ref{NEfinitize} and \ref{NWfinitize}) and create a new
polynomial $P'_\Box(x,y)$ from $P(x,y)$ for each now finite region. We will then
subdivide this finite region in order to prove that $P'_\Box(x,y) \geq 0$ on the region
in which it is defined. These new polynomials will be defined in the following manner
based on their corresponding region transformations.
\begin{figure}[h!]
\centering
    \includegraphics[scale=.45]{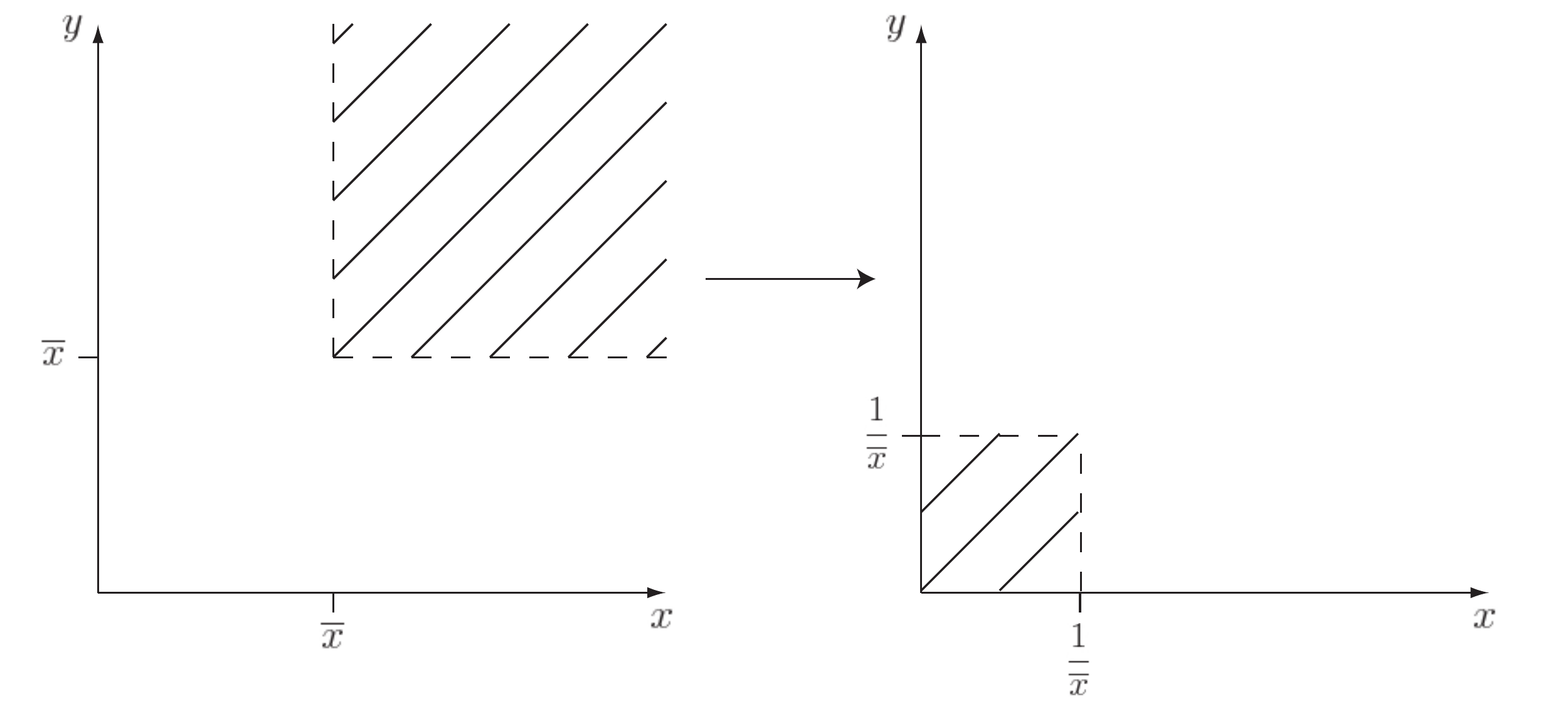}
\caption{Transforming NE to finite rectangle}\label{NEfinitize}
\end{figure}
\begin{figure}[h!]
\centering
    \includegraphics[scale=.45]{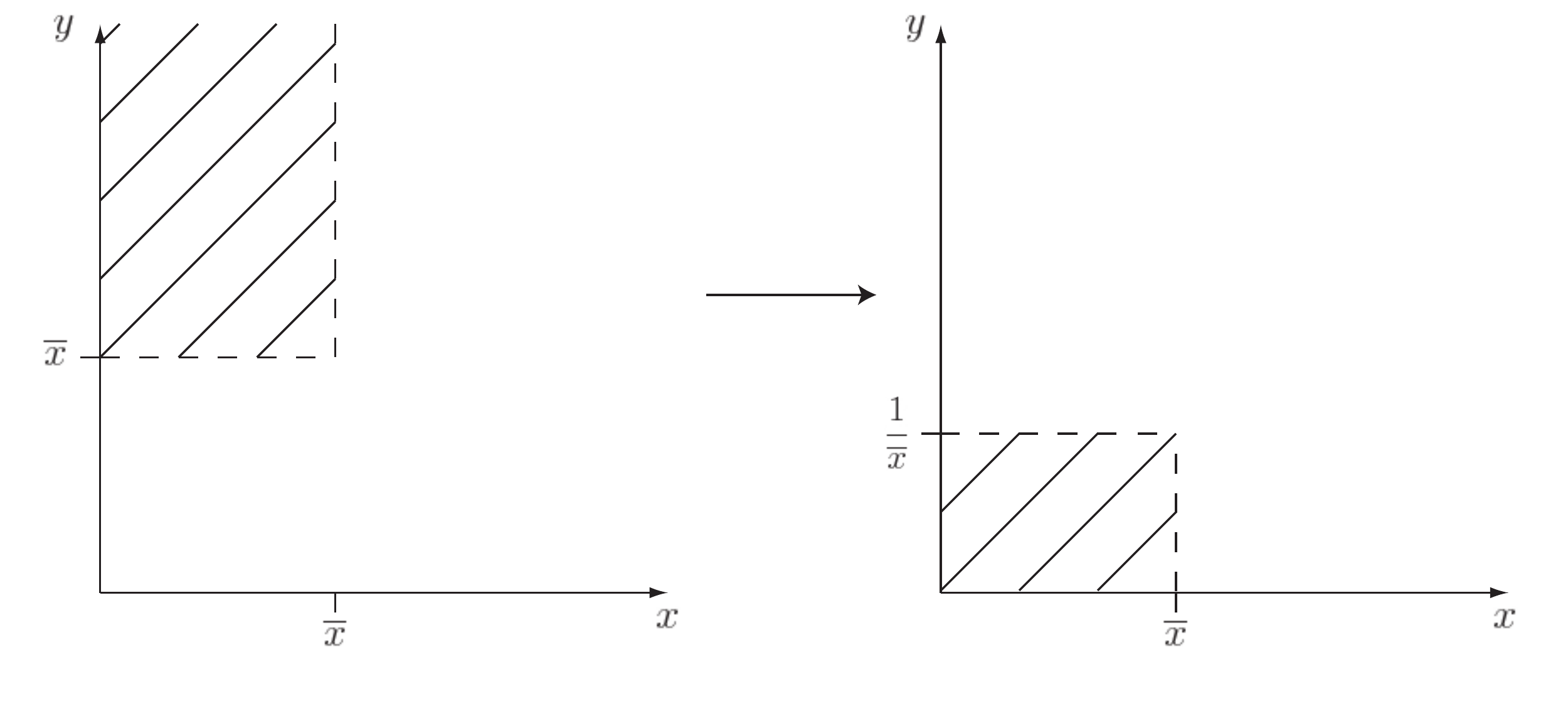}
\caption{Transforming NW to finite rectangle (SE similar by interchanging axes)}\label{NWfinitize}
\end{figure}

\begin{align}\label{fourPrimedPolys}
\begin{tabular}{rcll}
$P'_{NE}(x,y)$\hspace{-.35cm} &$=$& \hspace{-.35cm}$P\left( \frac{1}{x}, \frac{1}{y} \right) x^{d_x} y^{d_y}$ & restricted to $0<x,y\leq \frac{1}{\xbar}$\vspace{.1cm}\\
$P'_{NW}(x,y)$\hspace{-.35cm} &$=$& \hspace{-.35cm}$P\left( x,\frac{1}{y} \right)y^{d_y}$                     & restricted to $0\leq x\leq \xbar, 0< y \leq \frac{1}{\xbar}$ \vspace{.1cm}\\
$P'_{SE}(x,y)$\hspace{-.35cm} &$=$& \hspace{-.35cm}$P\left( \frac{1}{x},y \right)x^{d_x}$                     & restricted to $0< x\leq \frac{1}{\xbar}, 0\leq y \leq \xbar$ \vspace{.1cm}\\
$P'_{SW}(x,y)$\hspace{-.35cm} &$=$& \hspace{-.35cm}$P(x,y)$                                                   & restricted to $0\leq x\leq \xbar, 0< y \leq \frac{1}{\xbar}$
\end{tabular}
\end{align}

Along the lines of Proposition \ref{pos4regions} we can guarantee positivity of $P(x,y)$
given positivity of the related polynomials \eqref{fourPrimedPolys}.
\begin{prop}\label{pos4regionsPrime}
Let $P(x,y)$ be a polynomial, $d_x=\deg_{x}(P)$, $d_y=\deg_{y}(P)$, and $\xbar > 0$.
Consider the polynomials $P'_{NE}, P'_{SW}, P'_{NW}, P'_{SE}$ as defined in
\eqref{fourPrimedPolys}. If any one of these polynomials, generally denoted $P'_\Box$, is
positive on the region indicated in \eqref{fourPrimedPolys}, then $P(x,y)$ is positive on
the region $\Box$. For example, if $P'_{NW}(x,y) \geq 0$ on $0 \leq x \leq \xbar$ and $0
< y \leq \frac{1}{\xbar}$, then $P(x,y) \geq 0$ on the region $NW$, and similarly for the
other polynomials/regions.
\end{prop}
\begin{proof}
We will only see the proof for $P'_{NW}(x,y)$, the rest will follow in nearly the same
manner. Assume that $P'_{NW}(x,y) \geq 0$ for $0 \leq x \leq \xbar$ and $0 < y \leq
\frac{1}{\xbar}$. Then by definition of $P'_{NW}(x,y)$ we know that
\[P\left( x,\frac{1}{y} \right)y^{d_y} \geq 0 \quad \text{for }0 \leq x \leq \xbar, 0 < y \leq \frac{1}{\xbar}.\]
Since $y$ is strictly positive in the region in which $P'_{NW}(x,y)$ is defined, we can
cancel $y^{d_y}$ without reversing the inequality just as we did in the proof of
Proposition \ref{pos4regions}. Now, let $y' := \frac{1}{y}$ to see that
\[P\left( x, y'\right) \geq 0\]
for $0 \leq x \leq \xbar$ and $y' = \frac{1}{y} \geq \xbar$, which is precisely the
region $NW$. The other regions will follow by doing substitutions $x' := \frac{1}{x}$ and
$y' := \frac{1}{y}$ as necessary. Note that no work needs to be done for $SW$ since
$P'_{SW}(x,y) = P(x,y)$ and is defined in the region $SW$.
\end{proof}
Next, we need to see how to prove that $P'_\Box(x,y)\geq 0$ on the desired region. We
will do this by subdividing the domain of $P'_\Box$ into finitely many smaller
rectangles. Then, for each smaller rectangle, $S = \left\{ a \leq x \leq b, c \leq y \leq
d \right\}$, we transform it to $\posOrthP$ creating a corresponding polynomial,
$P''_S(x,y)$, and test criteria \posCoeffs and \subPoly to see whether this polynomial is
positive. See Figure \ref{rectTransf} for the transformation of a general rectangle, $S$,
to $\posOrthP$. Given this transformation, the polynomial, $P''_S$, is given by
\begin{align}\label{Prectangle}
P''_S(x,y) &= \sigma_{\frac{1}{b-a},\frac{1}{d-c}}\left(P'_\Box\left(\frac{1}{x}+a,\frac{1}{y}+c\right) x^{d'_x} y^{d'_y}\right)\notag\\
         &=P'_\Box\left(\frac{1}{x+\frac{1}{b-a}}+a,\frac{1}{y+\frac{1}{d-c}}+c\right) \left(x+\frac{1}{b-a}\right)^{d'_x} \left(y+\frac{1}{d-c}\right)^{d'_y}
\end{align}
where $P'_\Box$ is one of $P'_{NE}$, $P'_{NW}$, $P'_{SE}$, $P'_{SW}$, and $d'_z =
\text{degree of } z \text{ in } P'_\Box$ for $z=x,y$.
\begin{figure}[h!]
\centering
    \includegraphics[scale=.45]{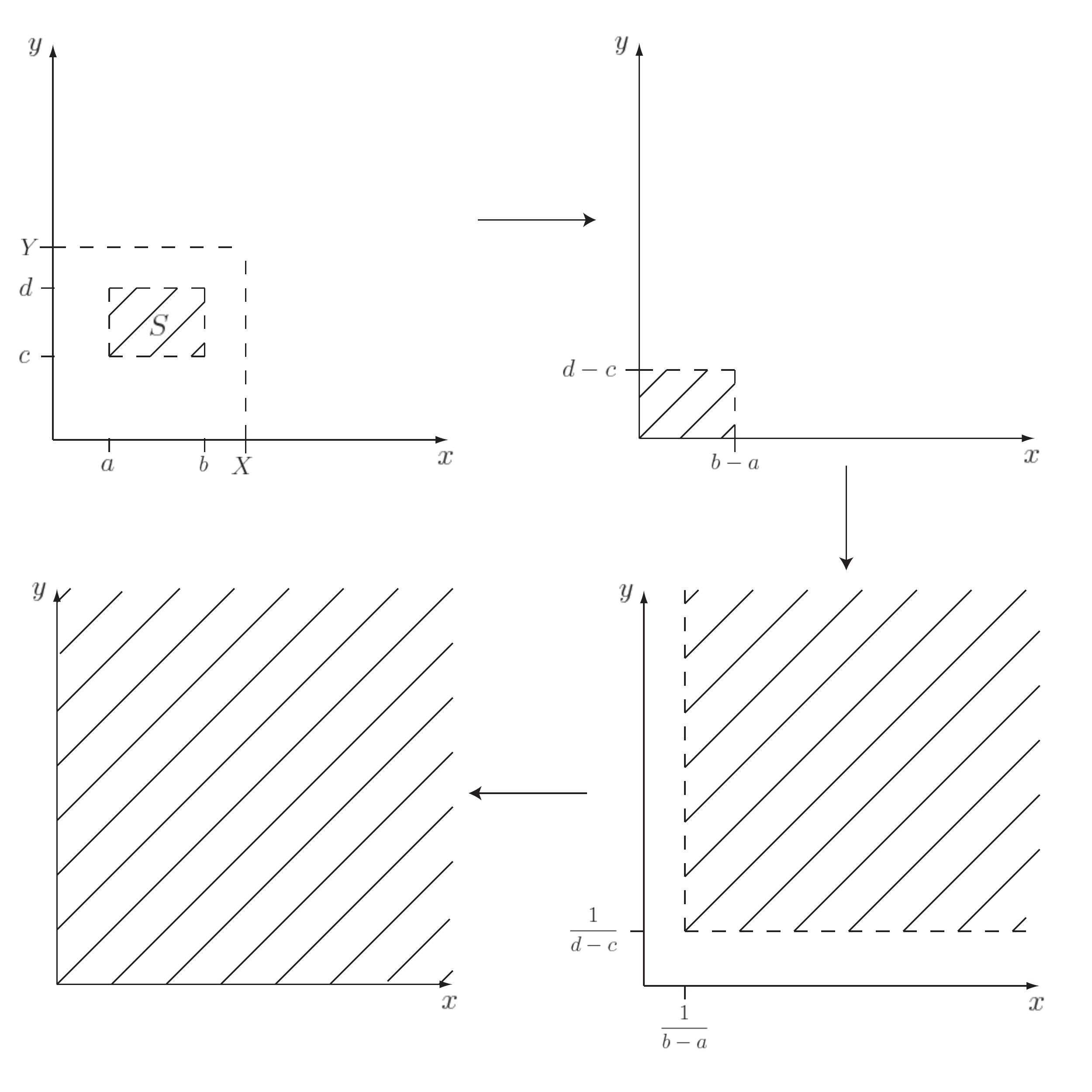}
\caption{Transforming general rectangle, $S$, to $\posOrthP$}\label{rectTransf}
\end{figure}
Before we see the canonical subdivision algorithm let us see why this $P''_S(x,y)$ will
give the desired result.
\begin{prop}\label{subDivProp}
Let $P'_\Box(x,y)$ be a polynomial, $d'_x=\deg_{x}(P'_\Box)$, $d'_y=\deg_{y}(P'_\Box)$,
$0 \leq a < b$, and $0 \leq c < d$. Consider the polynomial $P''_S$ as defined in
\eqref{Prectangle}. If $P''_S(x,y)$ is positive on $\posOrthP$, then $P'_\Box(x,y)$ is
positive on the rectangle $S = \left\{ a \leq x \leq b, c \leq y \leq d \right\}$.
\end{prop}
\begin{proof}
This proof follows the form of the proofs for Propositions \ref{pos4regions} and
\ref{pos4regionsPrime}. First, by definition of $P''(S)$, the fact that $P''_S(x,y) \geq
0$ for $x,y \geq 0$ means
\begin{align*}
P'_\Box\left(\frac{1}{x+\frac{1}{b-a}}+a,\frac{1}{y+\frac{1}{d-c}}+c\right) \left(x+\frac{1}{b-a}\right)^{d'_x} \left(y+\frac{1}{d-c}\right)^{d'_y} \geq 0.
\end{align*}
As in the previous proofs we may cancel $\left(x+\frac{1}{b-a}\right)^{d'_x}
\left(y+\frac{1}{d-c}\right)^{d'_y}$ without reversing the inequality. Let $x' =
x+\frac{1}{b-a}$ and $y' = y+\frac{1}{d-c}$, then
\begin{align*}
P'_\Box\left(\frac{1}{x'}+a,\frac{1}{y'}+c\right) \geq 0
\end{align*}
for $x' = x+\frac{1}{b-a} \geq \frac{1}{b-a}$ and $y' = y+\frac{1}{d-c} \geq
\frac{1}{d-c}$. Next, let $x'' = \frac{1}{x'}$ and $y'' = \frac{1}{y'}$. Making this
substitution yields
\[P'_\Box\left(x''+a,y''+c\right) \geq 0\]
for $0 < x'' = \frac{1}{x'} \leq b-a$ and $0 < y'' = \frac{1}{y'} \leq d-c$. For the
final substitution, let $x''' = x''+a$ and $y''' = y''+c$. Then
\[P'_\Box\left(x''', y'''\right) \geq 0\]
for $a < x''' = x''+a \leq b$ and $c<y''' = y''+c \leq d$. In other words, $P'_\Box
\left(x, y \right) \geq 0$ for $(x,y) \in S$.
\end{proof}

In principle any subdivision will work so long as we cover the entire finite rectangle.
However, since the goal is to program the algorithm we need to specify a canonical
subdivision. First we will simply divide into four equal regions. For each region we
perform the above steps (transform the region and polynomial, apply criteria \posCoeffs
and \subPoly\!\!). If we fail either criteria on a specific subregion, then we subdivide
that subregion into four again and repeat. We continue to do this until we pass
\posCoeffs or \subPoly\!\!\!, fail \lCoeff or \Const (and output \textbf{false}), or we
reach some stopping condition and output \textbf{FAIL}. A stopping condition could be
that we have subdivided $N$ times, for some large $N$.

Before we summarize the positivity algorithm in the general $n$-dimensional case we must
see the general \subPoly criteria. It was previously stated only in the case of 2
variables.
\begin{description}
\item[SubPoly-$n$:] If the only negative coefficients in $\Pbox$ (including the
    constant term) are on terms of the form $x_r x_s$ then we check whether the
    quadratic form \cite{La2004},
    \begin{align}\label{subP}
    \sum_{\stackrel{i,j=1}{i\leq j}}^n a_{i,j} x_i x_j
    \end{align}
    where $a_{i,j}$ are coefficients of their respective terms in $\Pbox$, is
    positive definite (i.e., is positive for all $\tup{x_1,\ldots,x_n} \neq
    \tup{0,\ldots,0}$) using its corresponding matrix. The symmetric coefficient
    matrix is defined as: $A = (a'_{i,j})$, where $a'_{i,j}=a'_{j,i} =
    \frac{1}{2}a_{i,j}$ for all $i\neq j$, and $a'_{i,i}=a_{i,i}$ \cite{OMeO1971}.
    Given this matrix, we can equivalently think of the quadratic form as $\mathbf{x}
    A \mathbf{x}^T$, where $\mathbf{x}=\tup{x_1,\ldots,x_n}$. Since $A$ is a
    symmetric matrix, we know from the spectral theorem that it is diagonalizable by
    an orthonormal matrix, $Q$, so we have that $A = Q D Q^{\top}$ where $D$ is a
    diagonal matrix. We can then rewrite the quadratic form as:
    \begin{align*}
    \mathbf{x} A \mathbf{x}^T &= \mathbf{x} Q D Q^{\top} \mathbf{x}^{\top}\\
                              &= \tilde{\mathbf{x}} D \tilde{\mathbf{x}}^{\top}.
    \end{align*}
    From this we easily see that the quadratic form is positive definite iff all
    eigenvalues of $A$ are positive (i.e., $A$ is positive definite). Then, if this
    quadratic ``sub-polynomial" of $\Pbox$ is positive, $\Pbox$ itself is positive
    (since the other coefficients are positive).
\end{description}

We are now ready to summarize the algorithm in the $n$-dimensional case. Assume we have a
polynomial $P \in \R[x_1,\ldots,x_n]$, and want to test whether $P(x_1,\ldots,x_n)\geq0$
for $(x_1,\ldots,x_n) \in \posOrth$.
\begin{enumerate}
\item\label{subdivide1} First cut $\posOrth$ into regions, similar to the $NW$, $NE$,
    $SW$, $SE$ regions. For each variable we have 2 possibilities for its domain
\[ 0 \leq x_i \leq \xbar \quad \text{or} \quad \xbar \leq x_i <\infty. \]
A region is defined by making a choice for each variable, thus we have $2^n$ regions.
The associated polynomial, $P_\Box$, for each region is then created by substituting
\begin{align*}
x_i \, \text{by} \, \left\{\begin{array}{ll}
                            x_i + \xbar & \text{if } \xbar \leq x_i <\infty\\
                            \frac{1}{x_i+\frac{1}{\xbar}} & \text{if } 0 \leq x_i \leq \xbar
                            \end{array}\right.
\end{align*}
in $P$, and then multiplying by $\left(x_i+\frac{1}{\xbar}\right)^{d_{x_i}}$ if $ 0
\leq x_i \leq \xbar$.
\item\label{test1} For each region we check our 4 criteria \posCoeffs\!\!\!,
    \subPoly\!\!-$n$, \lCoeff\!\!\!, and \Const\!\!. If all $2^n$ polynomials pass
    \posCoeffs or \subPoly\!\!-$n$ then we are done, and $P$ is positive on
    $\posOrth$. If any of the polynomials fail \lCoeff or \Const then we are also
    done because we know that there are values in $\posOrth$ for the variables which
    make $P$ negative. Otherwise, we continue on to step 3 for the regions which fail
    \posCoeffs and \subPoly\!\!-$n$.
\item\label{StepMakeFinite}Assume we have a specific region (domains for each
    variable), $R$, which failed step \ref{test1}. Then we create the polynomial
    $P'_R$ by substituting $\frac{1}{x_i}$ for $x_i$ in $P$ if $x_i$ is restricted to
    $\xbar \leq x_i \leq \infty$ in $R$, and then multiplying by $x_i^{d_{x_i}}$ for
    those variables which were substituted. This new polynomial will be restricted to
    the region $R'$, which is defined from $R$ in the following manner: if $0 \leq
    x_i \leq \xbar$ in $R$, then $x_i$ has the same restriction in $R'$; otherwise,
    $\xbar \leq x_i < \infty$ in $R$, and then $x_i$ is restricted to $0 < x_i \leq
    \frac{1}{\xbar}$ in $R'$. More formally, if $D = \{i: 0 \leq x_i \leq \xbar
    \text{ in }R\}$, and $\bar{D}=[n]\smsetminus D$ then
    \[R' = \left(\bigtimes_{i \in D} \left\{0 \leq x_i \leq \xbar\right\}\right) \times \left(\bigtimes_{i \in \bar{D}} \left\{0 < x_i \leq \frac{1}{\xbar}\right\}\right) \]
    \begin{enumerate}
    \item Subdivide $R'$ into $2^n$ equal regions, $S_j$, and for each region
        create the polynomial $P'_{S_j}$ in the same manner as
        \eqref{Prectangle}.
    \item Test positivity of $P'_{S_j}$ using criteria \posCoeffs\!\!\!,
        \subPoly\!\!-$n$, \lCoeff\!\!\!, and \Const\!\!\!. If $P'_{S_j}$ passes
        \posCoeffs or \subPoly\!\!-$n$ then we are done in region $S_j$ and can
        continue checking the rest of the subregions of $R'$. If $P'_{S_j}$ fails
        \lCoeff or \Const then we stop altogether because we know that there are
        values for the variables in $\posOrth$ which make $P$ negative.
        Otherwise, go back to step 3(a) with region $R'$ now replaced by $S_j$.
    \item If we have recursed more than $N$ times (for some choice of $N$), stop
        and output \textbf{FAIL}.
    \end{enumerate}
\end{enumerate}

Before going on to the proof-of-concept for a specific difference equation and
equilibrium let us see how to apply the polynomial positivity algorithm. We will see two
examples, one in which \subPoly must be used, and one where subdivisions are necessary.
\begin{ex}
Let $P(x,y) = x^2-xy+y^2$ and $\xbar = 1$. First we subdivide $\posOrthP$ into the for
regions $NE$, $SW$, $NW$, $SE$ and get the following polynomials:
\begin{align*}
P_{NE}(x,y) &= \sigma_{1}(P) = x^2-xy+y^2+x+y+1, \\
P_{SW}(x,y) &= \sigma_{1}\left(P\left( \frac{1}{x},\frac{1}{y} \right) x^{d_x} y^{d_y}\right) =
                x^2-xy+y^2+x+y+1,\\
P_{NW}(x,y) &= \sigma_{1,1}\left( P\left(\frac{1}{x}, y\right) x^{d_x}\right) =
                x^2 y^2+2 x^2 y+2 x y^2+x^2+3 x y+y^2+x+y+1,\\
P_{SE}(x,y) &= \sigma_{1, 1}\left( P\left( x,\frac{1}{y} \right)y^{d_y} \right) =
                x^2 y^2+2 x^2 y+2 x y^2+x^2+3 x y+y^2+x+y+1.
\end{align*}
The polynomials $P_{NW}(x,y) = P_{SE}(x,y)$ have all positive coefficients, so they
satisfy \posCoeffs\!\!. To see that $P_{NE}(x,y)$ (which equals $P_{SW}(x,y)$ in this
example) is positive we must test criteria \subPoly\!\!\!.

The only negative coefficient in $P_{NE}(x,y)$ is on the $xy$ term, so we look at the
sub-polynomial $x^2-xy+y^2$. The discriminant of this binary quadratic form is
$4\cdot1\cdot1-1^2=3$, which is positive as needed. So we see that $P_{NE}(x,y)$ (and
thus $P_{SW}(x,y)$) is positive by \subPoly\!\!\!.

In this example we don't have to do further subdivisions since $P_{NW}, P_{NE}, P_{SE},
P_{SW}$ are all positive. Therefore, we are done by Proposition \ref{pos4regions}.\Endex
\end{ex}

\begin{ex}
Let $P(x,y) = x^4 y-5 x^3 y+10 x^2 y+x+y$ and $\xbar = 1$. First we subdivide $\posOrthP$
into the for regions $NE$, $SW$, $NW$, $SE$ and get the following polynomials:
\begin{align*}
P_{NE}(x,y) &= \sigma_{1}(P) = x^4 y+x^4-x^3 y-x^3+x^2 y+2 x^2+9 x y+11 x+7 y+8,\\
P_{SW}(x,y) &= \sigma_{1}\left(P\left( \frac{1}{x},\frac{1}{y} \right) x^{d_x} y^{d_y}\right) =
               x^4+4 x^3+x^2 y+17 x^2+2 x y+21 x+y+8,\\
P_{NW}(x,y) &= \sigma_{1,1}\left( P\left(\frac{1}{x}, y\right) x^{d_x}\right) \\
               &= x^4 y+x^4+4 x^3 y+4 x^3+16 x^2 y+17 x^2+19 x y+21 x+7 y+8,\\
P_{SE}(x,y) &= \sigma_{1, 1}\left( P\left( x,\frac{1}{y} \right)y^{d_y} \right) =
               x^4-x^3+x^2 y+2 x^2+2 x y+11 x+y+8.
\end{align*}
In this example we see that $P_{SW}(x,y)$ and $P_{NW}(x,y)$ pass criteria \posCoeffs
since all coefficients are positive. For the other two regions we will need to subdivide
because the negative coefficients are not on the term $xy$.

Let us first examine $SE$. We need to create the polynomial $P'_{SE}(x,y)$ as in
\ref{fourPrimedPolys}:
\begin{align*}
P'_{SE}(x,y) &= P\left( \frac{1}{x},y \right)x^{d_x}\\
             &= x^4 y+10 x^2 y+x^2-5 x y+y\quad \text{restricted to } 0<x \leq 1, 0\leq y \leq 1.
\end{align*}
Then we subdivide the region $0<x \leq 1, 0\leq y \leq 1$ into four equal rectangles:
\begin{align*}
S_1 &= \left\{0 < x \leq \frac{1}{2}, 0\leq y \leq \frac{1}{2}\right\},\qquad
S_2 = \left\{0 < x \leq \frac{1}{2}, \frac{1}{2} \leq y \leq 1\right\},\\
S_3 &= \left\{\frac{1}{2} \leq x \leq 1, 0\leq y \leq \frac{1}{2}\right\},\qquad
S_4 = \left\{\frac{1}{2} \leq x \leq 1, \frac{1}{2} \leq y \leq 1\right\},
\end{align*}
and create four associated polynomials using \eqref{Prectangle}:
\begin{align*}
P''_{S_1}(x,y) =& x^4+3 x^3+x^2 y+6 x^2+4 x y+20 x+4 y+25,\\
P''_{S_2}(x,y) =& \frac{1}{2} x^4 y+2 x^4+\frac{3}{2} x^3 y+6 x^3+ 3 x^2 y+10 x^2+\\
                & +10 x y+32 x+\frac{25}{2} y+42,\\
P''_{S_3}(x,y) =& \frac{1}{4} x^4 y+\frac{25}{16}x^4+3 x^3 y+20 x^3+ 13 x^2 y+\\
                & +96 x^2+24 x y+196 x+16 y+144,\\
P''_{S_4}(x,y) =& \frac{25}{32} x^4 y+\frac{21}{8} x^4+10 x^3 y+34 x^3+48 x^2 y+\\
                & +166 x^2+98 x y+344 x+72 y+256.
\end{align*}
All four polynomials for the subdivision of $SE$ are positive by \posCoeffs\!\!\!,
therefore $P'_{SE}(x,y) \geq 0$, and by Proposition \ref{pos4regionsPrime} we see that
$P(x,y)\geq 0$ on the region $SE$.

Now let us look at $P_{NE}$. Create $P'_{NE}$ as indicated by \eqref{fourPrimedPolys}:
\begin{align*}
P'_{NE}(x,y) &= P\left( \frac{1}{x},\frac{1}{y} \right)x^{d_x}y^{d_y}\\
             &= 8 x^4 y+7 x^4+11 x^3 y+9 x^3+2 x^2 y+x^2-x y-x+y+1 \\
             &\hspace{2in}\text{restricted to } 0<x \leq 1, 0\leq y \leq 1.
\end{align*}
Subdivide the region $0<x \leq 1, 0\leq y \leq 1$ into the same $S_1$, $S_2$, $S_3$, and
$S_4$ as above, and create the polynomials $P''_{S_i}$, this time from $P'_{NE}$:
\begin{align*}
P''_{S_1}(x,y) =& x^4 y+3 x^4+7 x^3 y+21 x^3+19 x^2 y+58 x^2+\\
                & + 33 x y+105 x+37 y+120,\\
P''_{S_2}(x,y) =& \frac{3}{2} x^4 y+4 x^4+\frac{21}{2} x^3 y+28 x^3+29 x^2 y+78 x^2+\\
                & + \frac{105}{2} x y+144 x+60 y+166,\\
P''_{S_3}(x,y) =& \frac{37}{16} x^4 y+\frac{15}{2} x^4+\frac{115}{4} x^3 y+\frac{375}{4} x^3+142 x^2 y+\\
                & + 463 x^2+320 x y+1040 x+272 y+880,\\
P''_{S_4}(x,y) =& \frac{15}{4} x^4 y+\frac{83}{8} x^4+\frac{375}{8} x^3 y+130 x^3+\frac{463}{2} x^2 y+\\
                & + 642 x^2+520 x y+1440 x+440 y+1216.
\end{align*}
As before, all four subdivision polynomials pass \posCoeffs\!\!\!, and so they are
positive. Therefore, by Propositions \ref{pos4regions}, \ref{pos4regionsPrime}, and
\ref{subDivProp} we know that $P(x,y)\geq 0$ for $(x,y)~\in~\posOrthP$.\Endex
\end{ex}

\section{Proof of Concept}\label{ProofOfConcept} We have now seen the full algorithm to
prove GAS of equilibrium points of rational difference equations. However, there is no
reason a priori that this algorithm is applicable. It could be the case that no such $K$
(see Section \ref{GAStoPoly} for a definition of $K$) exists, and this algorithm would be
useless.

We will now see that this technique does, in fact, work to prove global asymptotic
stability of an equilibrium of a particular rational difference equation. The proof of
the following theorem to establish global asymptotic stability will go through the
procedure outlined in the previous sections.
\begin{thm}
For the rational difference equation
\begin{align}\label{RunningExample}
x_{n+1} = \frac{4+x_{n}}{1+x_{n-1}},
\end{align}
the equilibrium, $\xbar= 2$, is GAS.
\end{thm}
\begin{proof}
We will prove that $K=5$ satisfies \eqref{contrCriteria} from Theorem \ref{contrGAS}.
From the rational difference equation, the equilibrium $\xbar=2$, and $K=5$ we get the
polynomial, $P~:=~P_{\tup{2,2},5}(\tup{x_1,x_2})$, as defined in \eqref{polyToPos}:

\begin{align*}
P &= 25x_1^8x_2^4+340x_1^8x_2^3+1606x_1^8x_2^2+3060x_1^8x_2+2025x_1^8+60x_1^7x_2^5+
1158x_1^7x_2^4+\\
+&8460x_1^7x_2^3+28936x_1^7x_2^2+45848x_1^7x_2+27090x_1^7+71x_1^6x_2^6+
1418x_1^6x_2^5+\\
+&11229x_1^6x_2^4+53362x_1^6x_2^3+147345x_1^6x_2^2+207144x_1^6x_2+113103x_1^6+72x_1^5x_2^7+\\
+&1420x_1^5x_2^6+9012x_1^5x_2^5+20174x_1^5x_2^4+24716x_1^5x_2^3+74718x_1
^5x_2^2+163032x_1^5x_2+\\
+&108952x_1^5+47x_1^4x_2^8+1276x_1^4x_2^7+11120x_1^4x_2^6+25528x_1^4x_2^5-118780x_1^4x_2^4-\\
-&688300x_1^4x_2^3-1195361x_1^4x_2^2-790736x_1^4x_2-148969x_1^4+12x_1^3x_2^9+538x_1^3x_2^8+\\
+&7854x_1^3x_2^7+45864x_1^3x_2^6+53604x_1^3x_2^5-515564x_1^3x_2^4-2066454x_1^3x_2^3-\\
-&2469564x_1^3x_2^2-207576x_1^3x_2+833882x_1^3+x_1^2x_2^{10}+86x_1^2x_2^9+2109x_1^2x_2^8+\\
+&22070x_1^2x_2^7+102117x_1^2x_2^6+105526x_1^2x_2^5-695269x_1^2x_2^4-1867364x_1^2x_2^3+\\
+&785343x_1^2x_2^2+6256056x_1^2x_2+4716817x_1^2+4x_1x_2^{10}+198x_1x_2^9+3530x_1x_2^8+\\
+&29636x_1x_2^7+117218x_1x_2^6+136288x_1x_2^5-289440x_1x_2^4+253318x_1x_2^3+\\
+&5674806x_1x_2^2+11634024x_1x_2+7054300x_1+4x_2^{10}+148x_2^9+2145x_2^8+15348x_2^7+\\
+&53870x_2^6+69340x_2^5+30579x_2^4+801874x_2^3+3802411x_2^2+6262908x_2+3488704.
\end{align*}

The goal is to prove that this polynomial is positive when all variables are positive.
Recall that we created this polynomial by taking the numerator of
\[\myNorm{\calX-\Xbar}^2 - \myNorm{Q^5(\calX) - \Xbar}^2,\]
where $Q(\calX)$ is the map
\begin{align*}
Q\left(\left[\begin{array}{c}
               x_n\\
               x_{n-1}
               \end{array}\right]\right) = \left[\begin{array}{c}
                                        \frac{4+x_n}{1+x_{n-1}}\\
                                        x_{n}
                                        \end{array}\right].
\end{align*}
Now we run the polynomial positivity algorithm described in Section \ref{PosMethods} to
prove that this polynomial is positive. If the polynomial is positive when all variables
are positive then the equilibrium, $\xbar= 2$, is GAS for the original difference
equation by Theorem \ref{contrGAS}.

First we will prove that $P>0$ in the region $NE$. We make the polynomial $P_{NE}$ by
substituting $x_1= x_1+2$ and $x_2=x_2+2$ into $P$. See the Appendix in \cite{HoE2011}
for $P_{NE}$ and the rest of the polynomials as they will be omitted from this paper. Now
we need to prove that $P_{NE}>0$ in the region $\posOrthP$ except when all variables are
simultaneously zero. The only negative coefficient is on the term $x_1 x_2$, so we can
use the discriminant method. The binary quadratic form that we must show is positive
definite is
\begin{align*}
349366689 x_1^2-6980904 x_1 x_2+318700575 x_2^2.
\end{align*}
Its discriminant is $d= 445324725659927484$ which is positive, so by \subPoly $P_{NE}>0$
in $\posOrthP$.

Now we will prove $P>0$ in the region $NW$. Create the polynomial $P_{NW}$ by
substituting $x_1 = 1/x_1$, multiplying by $x_1^{d_{x_1}}= x_1^8$, and then translating
$x_1$ by $1/2$ to the left, and $x_2$ by $2$ to the left. All coefficients in $P_{NW}$
are positive, and the constant term is zero. There is no proper subset of the variables
for which setting them all equal zero yields the zero polynomial. Therefore, $P_{NW}$ is
zero only when all variables are zero, and so $P>0$ in $NW$.

Next, we will prove $P>0$ in the region $SE$. First make the polynomial $P_{SE}$ by
substituting $x_2= 1/x_2$, multiplying by $x_2^{d_{x_2}}= x_2^{10}$, and then translating
$x_1$ by $2$ to the left and $x_2$ by $1/2$ to the left. Now we need to prove that
$P_{SE}>0$ in the region $\posOrthP$ except when all variables are simultaneously zero.
All coefficients are positive, and the constant term is zero. There is no proper subset
of the variables for which setting them all equal zero yields the zero polynomial.
Therefore, $P_{SE}$ is zero only when all variables are zero, and then $P>0$ in the
region $SE$.

Finally, we must prove $P>0$ in the region $SW$. Make the polynomial $P_{SW}$ by
substituting $x_1= 1/x_1$ and $x_2= 1/x_2$, multiplying by $x_1^{d_{x_1}}= x_1^8$ and
$x_2^{d_{x_2}}= x_2^{10}$, and then translating both variables by $1/2$ to the left. Now
we need to prove that $P_{SW}>0$ in the region $\posOrthP$ except when all variables are
simultaneously zero. As in the region $NE$ the term $x_1 x_2$ has a negative coefficient
(and that is the only such coefficient), so we will use the discriminant method again.
The binary quadratic form that must be positive is
\begin{align*}
\frac{349366689}{16384} x_1^2-\frac{872613}{2048} x_1 x_2+\frac{318700575}{16384} x_2^2
\end{align*}
The discriminant is $d= \frac{111331181414981871}{67108864}$, which is positive. Then, by
\subPoly\!\!\!, $P_{SW}$ is positive in $\posOrthP$, so $P>0$ in the region $SW$.

Since $P>0$ in all four regions, $NE$, $NW$, $SE$, and $SW$, the $K$ value 5 is proven to
work for the rational difference equation $x_{n+1}= \frac{4+x_{n}}{1+x_{n-1}}$
\end{proof}

We can now see that the algorithm is indeed applicable. However, it wouldn't be possible
without programming the algorithm. For large $K$ values, even $K\geq 3$, the polynomials
are near impossible to deal with by hand. For this reason there is a maple package,
described below in Section \ref{MapleCodeGAS}.

\section{Results}\label{GASresults}

In this section we present the results that our algorithm can prove in full generality.
So far we have considered rational difference equations with specific numerical
coefficients. In this section, we consider the case where the coefficients are additional
variables, which are required to be positive. So the polynomial that we create is now a
polynomial in the variables $x_n, x_{n-1},\ldots,x_{n-k}$ as well as all of the
coefficient variables. We must point out that our algorithm will only apply when the
equilibrium can be expressed as a rational function of the coefficient variables. For the
proofs of the results found in this table see \cite{HoE2011}.

The equation numbers given in the following table match up with those in
\cite{CaELaG2008}, however the difference equations themselves may look different. The
parameters presented here are to guarantee that the equilibria will be rational functions
in the parameters.
\renewcommand{\arraystretch}{1.25}
\begin{longtable}{c|c|c|c}
Eqn \#    & $x_{n+1} =$                                                & Parameter Values          & Findings\\
\hline
\hline  2 & $\frac{M^2}{x_n}$                                          & $M\in \R$                 &$\xbar = |M|$ is not LAS\\
\hline  3 & $\frac{M^2}{x_{n-1}}$                                      & $M \in \R$ &$\xbar = |M|$ is not LAS\\
\hline  5 & \multirow{2}{*}{$\beta x_n$}                               & $0 \leq \beta < 1$        &$\xbar = 0$ is GAS\\
          &                                                            & $1 \leq \beta$            &$\xbar = 0$ is not LAS\\
\hline  9 & \multirow{2}{*}{$\gamma x_{n-1}$}                          & $0 \leq \gamma < 1$       & $\xbar = 0$ is GAS\\
          &                                                            & $1 \leq \gamma$           & $\xbar = 0$ is not LAS\\
\hline 17 & \multirow{2}{*}{$\frac{1}{4}\frac{M^2-1}{1+x_n}$}          & $M-1 > 0, M+1>0$          & $\xbar = \frac{1}{2}(M-1)$ is GAS\\
          &                                                            & $M-1 < 0, M+1<0$          & $\xbar = -\frac{1}{2}(M+1)$ is GAS\\
\hline 23 & \multirow{2}{*}{$\frac{\beta x_n}{1+x_n}$}                 & $0 < \beta \leq 1$        & $\xbar = 0$ is GAS\\
          &                                                            & $1 < \beta$               & $\xbar = \beta-1$ is GAS\\
\hline 29 & \multirow{2}{*}{$\frac{x_{n-1}}{A+x_n}$}                   & $0 < A < 1$               & $\xbar = 1-A$ is not LAS\\
          &                                                            & $1 < A$                   & $\xbar = 0$ is GAS\\
\hline 30 & \multirow{2}{*}{$\frac{x_{n-1}}{A+x_{n-1}}$} & $0 < A < 1$ & $\xbar = 1-A$ is GAS\\
          &                                                            & $1 < A$                   & $\xbar = 0$ is GAS\\
\hline 41 & $\alpha + \beta x_n$                                       & $0 \leq \beta < 1$        & $\xbar = \frac{\alpha}{1-\beta}$ is GAS \\
\hline 42 & \multirow{2}{*}{$q+\frac{1}{4}\frac{M^2-q^2}{x_n}$}        & $M-q < 0, M+q<0, q > 0$   & $\xbar = -\frac{1}{2}(M-q)$ is GAS\\
          &                                                            & $M-q > 0, M+q>0, q > 0$   & $\xbar = \frac{1}{2}(M+q)$ is GAS\\
\hline 65 & \multirow{2}{*}{$\frac{1}{4}\frac{M^2-q^2+4x_n}{1+q+x_n}$} & $M-q > 0, M+q > 0, q > -1$& $\xbar = \frac{1}{2}(M-q)$ is GAS\\
          &                                                            & $M-q < 0, M+q < 0, q > -1$& $\xbar = -\frac{1}{2}(M+q)$ is GAS\\
\hline 109& $\frac{x_{n-1}}{A+B x_n + x_{n-1}}$                        & $1 < A$                   & $\xbar = 0$ is GAS\\
\end{longtable}

In addition to the results in the above table our algorithm can be used to prove GAS of
many rational difference equations in which the coefficients have specific numerical
values. See the Web Books on the web page that accompanies this paper
\[ \texttt{http://math.rutgers.edu/\textasciitilde eahogan/GAS.html} \]

\section{Maple Code}\label{MapleCodeGAS}
In addition to the Web Books, there is a Maple package to accompany this paper which can
also be found on the above web page. The three most useful procedures are ProveK, Prove,
and WebBook. ProveK will use our algorithm to show that a given $K$ value works to prove
that the unique equilibrium of a particular rational difference equation is GAS. Prove
utilizes ProveK to find the $K$ value for a particular rational difference equation up to
a given threshold. WebBook takes in a rational difference equation with variables for
coefficients and proves GAS for a specified number of random choices for the
coefficients. There is a Help function (type Help() to see a list of all procedures, and
Help($\langle$procedure name$\rangle$) to get help on a specific procedure). In the help
for each procedure, a sample is given for how to use it.

\section{Conclusion} In Sections \ref{GAStoPoly} and \ref{PosMethods} we have seen both
parts of our new GAS algorithm: first reducing the problem to proving that a polynomial
is positive, and then proving polynomial positivity. Putting the two together we now have
a completely algorithmic approach to proving GAS of a given rational difference equation.

\vspace{\itemsep}
\noindent\textbf{Inputs:} \\
\indent $R$ - rational function in $k+1$ variables\\
\indent $\xbar$ - equilibrium, solution to $\xbar=R(\xbar,\ldots,\xbar)$\\
\indent $MaxK$ - a maximum $K$ value to try

\vspace{\itemsep}
\noindent\textbf{Outputs:}\\
\indent \textbf{true} if $\xbar$ is proven to be GAS for $x_{n+1}=R(x_{n},x_{n-1},\ldots,x_{n-k})$\\
\indent \textbf{false} if $\xbar$ is not LAS for $x_{n+1}=R(x_{n},x_{n-1},\ldots,x_{n-k})$\\
\indent \textbf{FAIL} if $MaxK$ was not high enough.

\vspace{\itemsep}
\noindent\textbf{Procedure:}
\begin{enumerate}
\item\label{CheckLAS} Check local asymptotic stability using the linearized stability
    theorem. If not LAS then output \textbf{false}. If LAS then continue to Step
    \ref{ConjK}.
\item\label{ConjK} Conjecture a $K$ value that satisfies Theorem \ref{contrGAS} using
    the procedure outlined in Section \ref{GAStoPoly}.
\item\label{ProvePos} Apply the $n$-dimensional polynomial positivity algorithm
    outlined at the end of Section \ref{PosMethods}. If the conjectured $K$ value was
    proven to work, output \textbf{true}. If the conjectured $K$ value was proven
    \emph{not} to work ($P_K$ failed \lCoeff or \Const\!\!), or the algorithm reached
    a recursion limit, continue to Step \ref{incrK}.
\item\label{incrK} If $K<MaxK$, increment $K$ by 1 and return to Step \ref{ProvePos}.
    If $K\geq MaxK$ then output \textbf{FAIL}.
\end{enumerate}

This algorithm now gives us a completely automatic proof machine for global asymptotic
stability. As was mentioned in the introduction, this problem has historically not been
approached in any kind of systematic fashion. Many of the theorems found in
\cite{CaELaG2008, KuMLaG2001} for proving GAS, were developed as generalizations of
techniques used to prove GAS of specific difference equations. This meant that given a
particular difference equation, proving its equilibrium is GAS would amount to trying to
apply various known theorems. Or, one may have to create a new theorem just to prove GAS
of one particular rational difference equation. There may not have been a clear cut path
leading to the proof. We believe that our new algorithm can serve as that path. Of
course, given a difference equation that is known to be GAS, our algorithm may not always
be able to prove it. However, we believe that it is much more widely applicable than any
one previously known theorem guaranteeing global asymptotic stability.

\section{Acknowledgements}
This material is based upon work supported by the U.S. Department of Homeland Security
under Grant Award Number 2007-ST-104-000006.  The views and conclusions contained in this
document are those of the authors and should not be interpreted as necessarily
representing the official policies, either expressed or implied, of the U.S. Department
of Homeland Security.

\bibliographystyle{myBibs}
\bibliography{AllBibs}

\end{document}